\documentclass[12pt]{amsart}
\usepackage{amssymb,latexsym}

\newtheorem*{maintheorem}{Theorem}
\newtheorem{theorem}{Theorem}[section]
\newtheorem{lemma}[theorem]{Lemma}
\newtheorem{proposition}[theorem]{Proposition}
\newtheorem{corollary}[theorem]{Corollary}

\theoremstyle{definition}
\newtheorem{definition}[theorem]{Definition}
\newtheorem{example}[theorem]{Example}
\newtheorem*{notations}{Notations}
\newtheorem{remark}[theorem]{Remark}
\newtheorem{remarks}[theorem]{Remarks}

\providecommand{\rparen}{)}

\newcounter{subenv}[theorem]

\newenvironment{subdefinition}{\begin{list}
  {\textit{\alph{subenv}}\rparen}
  {\topsep=0pt
   \itemsep=0pt
   \parsep=0pt
   \leftmargin=1cm
   \usecounter{subenv}
   \def\makelabel##1{\hss\llap{##1}}}}{\end{list}}
\newenvironment{subremarks}{\begin{list}
  {\textit{\alph{subenv}}\rparen}
  {\topsep=0pt
   \itemsep=0pt
   \parsep=0pt
   \leftmargin=1cm
   \usecounter{subenv}
   \def\makelabel##1{\hss\llap{##1}}}}{\end{list}}

\numberwithin{equation}{section}

\usepackage[a4paper,textwidth=13.5cm,textheight=23cm,centering]{geometry}
\usepackage[pdfstartview=XYZ]{hyperref}

\newcommand*{\N}{\mathbb{N}}
\newcommand*{\Z}{\mathbb{Z}}

\newcommand*{\agfield}{\mathbf{\Omega}}
\newcommand*{\agfieldnonzero}{\agfield^*}

\newcommand*{\ffield}[1]{\mathbb{F}_{\!#1}}
\newcommand*{\Fq}{\ffield{q}}
\newcommand*{\Fqr}{\ffield{q^r}}
\newcommand*{\Fqnonzero}{\ffield{q}^*}

\newcommand*{\ZnZ}{\Z/n\Z}

\newcommand*{\ZnZinv}{(\Z/n\Z)^\times}

\newcommand*{\Sn}{\mathfrak{S}_n}

\newcommand*{\cargroup}[1]{\hat{#1}}
\newcommand*{\cargroupG}{\cargroup{G}}
\newcommand*{\Fqnonzerocar}{\widehat{\mathbb{F}}^*_{\!q}}
\newcommand*{\Hom}{\mathop{\mathrm{Hom}}\nolimits}
\newcommand*{\caradd}{\varphi}
\newcommand*{\carmult}{\chi}
\newcommand*{\carmultbis}{\eta}
\newcommand*{\carmulttriv}{\mathbf{1}}
\newcommand*{\gausssum}[1]{G(\caradd,#1)}
\newcommand*{\jacobisum}{J}

\newcommand*{\dividesnot}{\nmid}
\newcommand*{\card}[1]{\lvert#1\rvert}

\newcommand*{\parameter}{\psi}
\newcommand*{\dworkhypersurface}[1][\parameter]{X_{#1}}
\newcommand*{\mirrordworkhypersurface}{Y_{\parameter}}
\newcommand*{\hypergeomvar}[1][\lambda]{H_{#1}}
\newcommand*{\quintichypersurface}{\mathcal{M}_{\parameter}}
\newcommand*{\curveA}{\mathcal{A}_{\parameter}}
\newcommand*{\curveB}{\mathcal{B}_{\parameter}}

\newcommand*{\varaff}[1]{\mathbb{A}^{\!#1}}
\newcommand*{\varprojFq}[1]{\mathbb{P}^{#1}_{\!\Fq}}
\newcommand*{\varaffFq}[1]{\mathbb{A}^{\!#1}_{\Fq}}
\newcommand*{\funzeta}[2]{Z_{#1}(#2)}

\newcommand*{\nbzerosaffdworkhypersurface}[1][\parameter]{\nu_q(\dworkhypersurface[#1])}
\newcommand*{\nbzerosaffdworkhypersurfaceallnotzero}[1][\parameter]{\nu_q^*(\dworkhypersurface[#1])}

\newcommand*{\coeffnbzerosempty}{\beta}
\newcommand*{\coeffnbzeros}[3]{\coeffnbzerosempty_{(#1),#2,#3}}

\newcommand*{\thrm}[1]{Theorem~#1}
\newcommand*{\pg}[1]{page~#1}
\newcommand*{\sctn}[1]{\S#1}
\newcommand*{\subsctn}[1]{\S#1}
\newcommand*{\subsctns}[1]{\S\S#1}
\newcommand*{\eqtn}[1]{Eq.~(#1)}
\newcommand*{\dfntn}[1]{Definition~#1}
\newcommand*{\rmrk}[1]{Remark~#1}
\newcommand*{\frml}[1]{Formula~#1}
\newcommand*{\lmm}[1]{Lemma~#1}
\newcommand*{\crllr}[1]{Corollary~#1}
\newcommand*{\prpstn}[1]{Proposition~#1}

\newcommand*{\tensor}[1][]{\otimes_{#1}}

\newcommand*{\set}[1]{\{#1\}}
\newcommand*{\setst}[2]{\{#1\mid#2\}}

\newcommand*{\classZnZ}[1]{[#1]}
\newcommand*{\classSn}[1]{\langle#1\rangle}
\newcommand*{\classSnZnZinv}[1]{\mkern1mu\underline{\mkern-1mu#1\mkern-1mu}\mkern1mu}

\newcommand*{\multSn}[1]{\gamma_{#1}}
\newcommand*{\multZnZinv}[1]{K_{#1}}

\newcommand*{\SaQellbarprime}[1][s]{S_{#1}'}
\newcommand*{\SaQellbar}[1][s]{S_{#1}}
\newcommand*{\SaQell}[1][s]{S_{\mkern1mu\overline{\mkern-1mu#1\mkern-1mu}\mkern1mu}}

\newenvironment{system}{\begin{cases}}{\end{cases}}

\newcommand{\integerinterval}[2]{[\![#1;#2]\!]}

\begin{document}
\baselineskip=17pt

\title[Factorisation of the zeta functions of Dwork hypersurfaces]{An explicit factorisation of the zeta functions of Dwork hypersurfaces}

\author[P.~Goutet]{Philippe Goutet}

\address{Institut de Math\'ematiques de Jussieu\\
case 247\\
4 place Jussieu\\
75252 Paris Cedex 05, France}

\email{goutet@math.jussieu.fr}

\date{}

\begin{abstract}
Let $\Fq$ be a finite field with $q$ elements, $\parameter$ a non-zero element of $\Fq$, and $n$ an integer $\geq 3$ prime to $q$. The aim of this article is to show that the zeta function of the projective variety over $\Fq$ defined by $\dworkhypersurface \colon x_1^n+\dots+x_n^n - n \parameter x_1\dots x_n=0$ has, when $n$ is prime and $\dworkhypersurface$ is non singular (i.e. when $\parameter^n \neq 1$), an explicit decomposition in factors coming from affine varieties of odd dimension $\leq n-4$ which are of hypergeometric type. The method we use consists in counting separately the number of points of $\dworkhypersurface$ and of some varieties of the preceding type and then compare them. This article answers, at least when $n$ is prime, a question asked by D.~Wan in his article ``Mirror Symmetry for Zeta Functions''.
\end{abstract}



\maketitle

\section{Introduction}

Let $n$ be an integer $\geq 3$ and $\Fq$ a finite field of characteristic $p \dividesnot n$. We consider the family of hypersurfaces of $\varprojFq{n-1}$ defined by
\[\dworkhypersurface \colon x_1^n+\dots+x_n^n - n \parameter x_1\dots x_n=0, \qquad \text{(Dwork family)}\]
where $\parameter \in \Fq$ is a non-zero parameter. We will make the assumption that $\dworkhypersurface$ is non-singular, i.e. that $\parameter^n \neq 1$. We denote by $\card{\dworkhypersurface(\Fqr)}$ the number of points of $\dworkhypersurface$ over an extension $\Fqr$ of degree $r$ of $\Fq$; the zeta function of $\dworkhypersurface$ is defined by
\[\funzeta{\dworkhypersurface/\Fq}{t} = \exp\biggl(\sum_{r=1}^{+\infty}{\card{\dworkhypersurface(\Fqr)}\mathinner{\frac{t^r}{r}}}\biggr).\]

When $q \equiv 1 \mod n$ (see \cite[\thrm{7.2} \pg{174}]{Wan.mirror}) and when $n$ is prime (see \cite[\thrm{9.5} \pg{179}]{Haessig.(n+1)-adic}), it is possible to show that the zeta function of $\dworkhypersurface$ takes the form
\[\funzeta{\dworkhypersurface/\Fq}{t} = \frac{\bigl(Q(t,\parameter)R(q^\rho t^\rho,\parameter)\bigr)^{(-1)^{n-1}}}{(1-t)(1-qt)\dots(1-q^{n-2}t)},\]
where $\rho$ is the order of $q$ in $\ZnZinv$.

In this formula, $Q(t,\parameter)$ is a polynomial with integer coefficients of degree $n-1$. As proved by D.~Wan (see \cite[\sctn{7}, \eqtn{14}, \pg{173}]{Wan.mirror}), this factor comes from the zeta function of the quotient $\mirrordworkhypersurface$ of $\dworkhypersurface \tensor \ffield{q^\rho}$ by the group $\setst{(\zeta_1,\dots,\zeta_n) \in \ffield{q^\rho}}{\zeta_i^n = 1\text{, }\zeta_1\dots\zeta_n=1}$ (Wan calls $\mirrordworkhypersurface$ a ``singular mirror'' of $\dworkhypersurface$):
\[\funzeta{\mirrordworkhypersurface/\Fq}{t} = \frac{Q(t,\parameter)^{(-1)^{n-1}}}{(1-t)(1-qt)\dots(1-q^{n-2}t)}.\]
A simple equation of $\mirrordworkhypersurface$ is $(y_1+\dots+y_n)^n = (n\parameter)^n y_1\dots y_n$.

The factor $R(t,\parameter)$ is a polynomial with integer coefficients of degree
\[\frac{(n-1)^n+(-1)^n(n-1)}{n} - (n-1)\]
whose roots have absolute values $q^{-(n-4)/2}$. We are interested in describing the factorisation of $R$; two approaches are possible: either predict, from a theoretical point of view, the existence of a factorisation of $R$, or look for explicit varieties with factors in their zeta functions appearing in $R$. Concerning the first approach, we refer to \cite{Kloosterman.zeta.monomial.fermat}. The second approach is raised by Wan in \cite[\sctn{7}, \pg{175}]{Wan.mirror} who mentions that it has been solved for $n=3$, $n=4$ (Dwork) and $n=5$ (Candelas, de la Ossa, and Rodriguez Villegas); a~recent article of Katz \cite{Katz.another.look.dwork.family} also talks about the subject from a different angle\footnote{His results are in terms of traces of the Frobenius of the toric hypersurfaces $x_1\dots x_n = \lambda y_1\dots y_m$ over a hypergeometric sheave.}.

The aim of this article is to handle the case where $n$ is a prime number $\geq 5$ by using only properties of Gauss sums; the fact that $n$ is prime allows to restrict to the case $q \equiv 1 \mod n$ in view of Haessig's result \cite[\thrm{9.5}, \pg{179}]{Haessig.(n+1)-adic} that, when $n$ is prime,
\[R(qt,\parameter) = R_{\dworkhypersurface/\ffield{q^\rho}}(q^\rho t^\rho,\parameter)^{1/\rho},\]
where $\rho$ is the order of $q$ in $\ZnZinv$. More precisely, if we define $N_R(q^r)$ by $R(t,\parameter) = \exp\bigl(\sum_{r=1}^{+\infty}{N_R(q^r)} \frac{t^r}{r}\bigr)$, we will show the following result (\thrm{\ref{result:conclusion}} \pg{\pageref{result:conclusion}}).

\begin{maintheorem}
Let $n$ be a prime number $\geq 5$ such that $q \equiv 1 \mod n$. We can write
\begin{equation}\label{formula:AIM}
N_R(q^r) = q^{\frac{n-5}{2}}N_1(q^r) + q^{\frac{n-7}{2}}N_3(q^r) + \dots + N_{n-4}(q^r),
\end{equation}
where each $N_d(q^r)$ is a sum of some $\card{H_{d,i}(q^r)} - (q-1)^{l-1}q^{d+1-l}$, the $H_{d,i}$ being varieties of $\varaffFq{d+2}$ of hypergeometric type of odd dimension equal to $d$ with $1 \leq d \leq n-4$ (their equations are explicitly given in \subsctn{\ref{subsection:nb.points.link:link}} \pg{\pageref{subsection:nb.points.link:link}}).
\end{maintheorem}
This equality in terms of number of points translates into a factorisation of the polynomial $R$ in terms of the zeta function of the preceding $H_{d,i}(q^r)$.

This article is organised as follows. In \sctn{\ref{section:gauss.sums}}, we recall the formulas concerning Gauss and Jacobi sums we will need in the rest of the article. In \sctn{\ref{section:nb.points.hypergeom}}, we compute, in terms of Gauss sums, the number of points of some varieties of hypergeometric type thanks to a method similar to the one Koblitz used in \cite{Koblitz.nb.points}. In \sctn{\ref{section:nb.points.dwork}}, we recall the formula for the number of points of $\dworkhypersurface$, and in \sctn{\ref{section:nb.points.link}}, we compare this formula with those from \sctn{\ref{section:nb.points.hypergeom}}. Finally, in \sctn{\ref{section:nb.points.examples}}, we detail the cases $n=5$ (already treated by Candelas, de la Ossa, and Rodriguez Villegas in \cite{CdlORV.II}) and $n=7$. The assumptions that $n$ is prime and that $q \equiv 1 \mod n$ will only be used starting from \sctn{\ref{section:nb.points.link}} and \subsctn{\ref{subsection:nb.points.dwork:formulas}} respectively.

Let us mention to finish that our method does not give a geometric link between $\dworkhypersurface$ and the varieties of hypergeometric type we consider.

\section{Gauss and Jacobi sums formulas}\label{section:gauss.sums}

In all this \sctn{\ref{section:gauss.sums}}, $\Fq$ will be a finite field with $q$ elements.

Let $\agfield$ be an algebraically closed field of characteristic zero, $G$ a finite abelian group and $\cargroupG = \Hom(G,\agfieldnonzero)$ its character group. Let us recall the following orthogonality formula:
\begin{equation}\label{formula:orthogonality:gen}
\frac{1}{\card{G}}\sum_{\caradd \in \cargroup{G}}{\caradd(g)} = \begin{cases} 1 & \text{if $g = e$,} \\ 0 & \text{if $g \neq e$,} \end{cases}
\end{equation}
where $e$ is the neutral element of $G$. In the following, we will use this formula when $G = \Fq$ or $G = \Fqnonzero$.\bigskip

Let us now fix a non-trivial additive character $\caradd \colon \Fq \to \agfieldnonzero$.

\begin{proposition}[Orthogonality formula]
\begin{equation}\label{formula:orthogonality:car.add}
\frac{1}{q}\sum_{a \in \Fq}{\caradd(ax)} = \begin{cases} 1 & \text{if $x = 0$,} \\ 0 & \text{if $x \neq 0$.} \end{cases}
\end{equation}
\end{proposition}

\begin{proof}
This results from \frml{\eqref{formula:orthogonality:gen}} above and the fact that every additive character is of the form $x \mapsto \caradd(ax)$ for some $a \in \Fq$.
\end{proof}

\begin{definition}[Gauss sums]
If $\carmult \colon \Fqnonzero \to \agfieldnonzero$ is a multiplicative character, let $\gausssum{\carmult}$ be the Gauss sum
\[\gausssum{\carmult} = \sum_{x \in \Fqnonzero}{\caradd(x)\carmult(x)}.\]
If $\carmulttriv$ is the trivial character of $\Fqnonzero$, we have $\gausssum{\carmulttriv} = -1$.
\end{definition}

\begin{proposition}[Reflection formula]
If $\carmult$ is a non-trivial character of $\Fqnonzero$,
\begin{equation}\label{formula:reflection}
\gausssum{\carmult} \gausssum{\carmult^{-1}} = \carmult(-1) q.
\end{equation}
\end{proposition}

\begin{proof}
Let us recall the proof of this simple property (see also \cite[\thrm{1.1.4~(a)}, \pg{10}]{BEW.gauss.jacobi.sums}). We have
\[\gausssum{\carmult} \gausssum{\carmult^{-1}} = \sum_{x,y \in \Fqnonzero}{\caradd(x+y)\carmult(\tfrac{x}{y})}.\]
Making the change of variable $x = yz$, we obtain
\begin{align*}
\gausssum{\carmult} \gausssum{\carmult^{-1}}
& = \sum_{y,z \in \Fqnonzero}{\caradd(y(1+z))\carmult(z)} \\
& = \carmult(-1)(q-1) + \sum_{\substack{z \in \Fqnonzero\text{, }z \neq -1}}{\biggl(\sum_{y \in \Fqnonzero}{\caradd(y(1+z))}\biggr)\carmult(z)}.
\end{align*}
We conclude by making the change of variable $y' = y(1+z)$ and by using an orthogonality formula.
\end{proof}

\begin{proposition}[Multiplication formula]
Let $d \geq 1$ be an integer dividing $q-1$. If $\carmultbis$ is a character of $\Fq^*$,
\begin{equation}\label{formula:multiplication}
\frac{\gausssum{\carmultbis^d}}{\prod_{\carmult^d = \carmulttriv}{\gausssum{\carmultbis\carmult}}} = \frac{\carmultbis(d)^d}{\prod_{\substack{\carmult^d = \carmulttriv \\ \carmult \neq \carmulttriv}}{\gausssum{\carmult}}}.
\end{equation}
\end{proposition}

\begin{proof}
This seemingly simple formula does not seem to admit an elementary proof; we refer the reader to \cite[\thrm{11.3.5} \pg{355}]{BEW.gauss.jacobi.sums} for additional details.
\end{proof}

\begin{definition}[Jacobi sums]\label{definition:jacobi.sums}
If $(\carmult_1,\dots,\carmult_r)$ is a finite sequence of characters of $\Fqnonzero$, we define
\[\jacobisum(\carmult_1,\dots,\carmult_r) = \sum_{\substack{x_1,\dots,x_r\in\Fqnonzero\\x_1+\dots+x_r=1}} {\carmult_1(x_1) \dots \carmult_r(x_r)}.\]
\end{definition}

\begin{proposition}[Link with Gauss sums]
If $\carmult_1$, \dots, $\carmult_r$ are characters of $\Fqnonzero$ not all trivial,
\begin{equation}\label{formula:jacobi.sums}
\jacobisum(\carmult_1,\dots,\carmult_r) = \begin{cases} \displaystyle \frac{1}{q}\frac{\gausssum{\carmult_1} \dots \gausssum{\carmult_r}}{\gausssum{\carmult_1\dots \carmult_r}} & \text{if $\carmult_1\dots\carmult_r = \carmulttriv$,}\\[9pt]
\displaystyle \phantom{\frac{1}{q}}\frac{\gausssum{\carmult_1} \dots \gausssum{\carmult_r}}{\gausssum{\carmult_1\dots \carmult_r}} & \text{if $\carmult_1\dots\carmult_r \neq \carmulttriv$.}
\end{cases}
\end{equation}
\end{proposition}

\begin{proof}
Let us briefly recall the proof (see also \cite[\thrm{10.3.1}, \pg{302}]{BEW.gauss.jacobi.sums}). The additive convolution of the functions $\carmult_1$, \dots, $\carmult_r$ is defined by
\[(\carmult_1 * \dots * \carmult_r)(a) = \sum_{\substack{x_1+\dots+x_r=a\\x_i \in \Fqnonzero}}{\carmult_1(x_1) \dots \carmult_r(x_r)}.\]
It is equal to $(\carmult_1\dots\carmult_r)(a) J(\carmult_1\dots\carmult_r)$ when $a \neq 0$. To compute the value when $a = 0$, we notice that the sum of $(\carmult_1 * \dots * \carmult_r)(a)$ over $a \in \Fq$ is $0$ since at least one of the $\carmult_i$ is non trivial. Thus, $(\carmult_1*\dots*\carmult_r)(0)$ is $0$ if $\carmult_1\dots\carmult_r \neq \carmulttriv$ and is $-(q-1)\jacobisum(\carmult_1,\dots,\carmult_r)$ if $\carmult_1\dots\carmult_r = \carmulttriv$. Moreover,
\[\prod_{i=1}^{r}{\gausssum{\carmult_i}} = \sum_{a\in\Fq}{\caradd(a)(\carmult_1*\dots*\carmult_r)(a)},\]
and so
\[\prod_{i=1}^{r}{\gausssum{\carmult_i}} = \jacobisum(\carmult_1,\dots,\carmult_r) \times \begin{cases} \gausssum{\carmult_1\dots\carmult_n} & \text{if $\carmult_1\dots\carmult_r \neq \carmulttriv$,} \\ \gausssum{\carmulttriv}-(q-1) & \text{if $\carmult_1\dots\carmult_r = \carmulttriv$,} \end{cases}\]
which shows the result.
\end{proof}

\begin{proposition}[Fourier inversion formula]
For every map $f \colon \Fq^* \to \agfield$,
\begin{equation}\label{formula:Fourier.inversion}
\forall x \in \Fq^*, \quad f(x) = \frac{1}{q-1}\sum_{\carmultbis \in \Fqnonzerocar}{\biggl(\sum_{y \in \Fq^*\vphantom{\carmultbis \in \Fqnonzerocar}}{f(y)\carmultbis^{-1}(y)}\biggr) \carmultbis(x)}.
\end{equation}
\end{proposition}

\begin{proof}
It is a direct consequence of the orthogonality formulas for the characters of the abelian group $\Fq^*$.
\end{proof}

\begin{corollary}
If $x \in \Fqnonzero$,
\begin{equation}\label{formula:caradd=gauss.sums}
\caradd(x) = \frac{1}{q-1}\sum_{\carmultbis \in \Fqnonzerocar}{\gausssum{\carmultbis^{-1}}\carmultbis(x)}.
\end{equation}
\end{corollary}

\section{Number of points of some varieties of hypergeometric type}\label{section:nb.points.hypergeom}

In all of \sctn{\ref{section:nb.points.hypergeom}}, $n$ will be an integer $\geq 2$ and $\Fq$ a finite field with $q$ elements.

\subsection{Computation of the number of points}\label{subsection:nb.points.hypergeom:formulas}

We consider here some affine varieties of hypergeometric type for which we compute the number of points by using Gauss sums and taking inspiration from Koblitz \cite[\sctn{5}]{Koblitz.nb.points}.

\begin{theorem}\label{result:nb.points.hyper.var}
Let $k \geq l \geq 2$ be two integers and $\lambda \in \Fqnonzero$ a parameter; we denote by $\hypergeomvar \subset \varaff{k+1}$ the affine variety defined by
\[\begin{cases} y^n = x_1^{\alpha_1} \dots x_k^{\alpha_k} (1-x_1)^{\beta_1} \dots (1-x_{l-1})^{\beta_{l-1}} (1-x_l-\dots-x_k)^{\beta_l} \\ \lambda x_1 \dots x_l = 1\end{cases}\]
where $\alpha_i$ and $\beta_i$ are integers $\geq 1$. The number of points of $\hypergeomvar$ over $\Fq$ is
\[\card{\hypergeomvar(\Fq)} = (q-1)^{l-1} q^{k-l} + \sum_{\substack{\carmult^n=\carmulttriv\\\carmult\neq\carmulttriv}}{\frac{1}{q-1} \sum_{\carmultbis}{\mathinner{N_{\lambda,\carmult,\carmultbis}} \carmultbis(\lambda)}},\]
where
\begingroup\footnotesize\[N_{\lambda,\carmult,\carmultbis} = \frac{1}{q^\nu} \frac{\gausssum{\carmult^{\alpha_1}\carmultbis} \dots \gausssum{\carmult^{\alpha_l}\carmultbis} \, \gausssum{\carmult^{\beta_1}} \dots \gausssum{\carmult^{\beta_l}} \gausssum{\carmult^{\alpha_{l+1}}} \dots \gausssum{\carmult^{\alpha_{k}}}} {\gausssum{\carmult^{\alpha_1+\beta_1}\carmultbis} \dots \gausssum{\carmult^{\alpha_{l-1}+\beta_{l-1}}\carmultbis} \gausssum{\carmult^{\alpha_l+\dots+\alpha_{k}+\beta_l}\carmultbis}},\]\endgroup
with $\nu$ denoting the number of trivial characters among those appearing in the denominator (namely, $\carmult^{\alpha_j+\beta_j}\carmultbis$ for $1 \leq j \leq l-1$ and $\carmult^{\alpha_l+\dots+\alpha_{k}+\beta_l}\carmultbis$).
\end{theorem}

\begin{proof}
To simplify, we shall write $y^n = Q(x_1,\dots,x_k)$ for the first equation defining $\hypergeomvar$. We have
\begin{align*}
\card{\hypergeomvar(\Fq)}
& = \sum_{\substack{x \in \Fq^k\text{, }y \in \Fq \\ y^n = Q(x) \\ \lambda x_1 \dots x_l = 1}}{1} = \sum_{\substack{x \in \Fq^k \\ \lambda x_1 \dots x_l = 1 \vphantom{Q(x)}}}{\sum_{\substack{y \in \Fq \vphantom{\Fq^k} \\ y^n = Q(x)}}{1}},
\end{align*}
with
\[\card{\setst{y \in \Fq}{y^n = z}} = \begin{cases} 1 & \text{if $z=0$,} \\ \smash{1 + \sum\limits_{\substack{\carmult^n = \carmulttriv \\ \carmult\neq\carmulttriv}}{\carmult(z)}} & \text{if $z \neq 0$,}\end{cases} \vrule height 0pt depth 24pt width 0pt\]
and thus
\begin{align*}
\card{\hypergeomvar(\Fq)}
& = \sum_{\substack{x \in \Fq^k \\ \lambda x_1 \dots x_l = 1 \\ Q(x) = 0}}{1} + \sum_{\substack{x \in \Fq^k \\ \lambda x_1 \dots x_l = 1 \\ Q(x) \neq 0}}{\biggl( 1+\sum_{\substack{\carmult^n = \carmulttriv \\ \carmult \neq \carmulttriv}}{\carmult(Q(x))}\biggr)} \\
& = \sum_{\substack{x \in \Fq^k \\ \lambda x_1 \dots x_l = 1}}{1} + \sum_{\substack{x \in \Fq^k \\ \lambda x_1 \dots x_l = 1 \\ Q(x) \neq 0}}{\sum_{\substack{\carmult^n = \carmulttriv \\ \carmult \neq \carmulttriv}}{\carmult(Q(x))}} \\
& = (q-1)^{l-1} q^{k-l} + \sum_{\substack{\carmult^n = \carmulttriv \\ \carmult \neq \carmulttriv}}{\sum_{\substack{x \in \Fq^k \\ \lambda x_1 \dots x_l = 1 \\ Q(x) \neq 0}}{\carmult(Q(x))}} \\
& = (q-1)^{l-1} q^{k-l} + \sum_{\substack{\carmult^n = \carmulttriv \\ \carmult \neq \carmulttriv}}{\sum_{\substack{x \in \Fq^k \\ Q(x) \neq 0}}{\carmult(Q(x))\delta_{\lambda x_1 \dots x_l, 1}}},
\end{align*}
where $\delta_{z,z'}$ is the Kronecker delta ($=1$ if $z=z'$ and $=0$ otherwise). Because
\[\forall z,z' \in \Fqnonzero, \quad \delta_{z,z'} = \frac{1}{q-1} \sum_{\carmultbis \in \Fqnonzerocar}{\carmultbis(\tfrac{z}{z'})},\]
we may write
\begin{align*}
\card{\hypergeomvar(\Fq)}
& = (q-1)^{l-1} q^{k-l} \\
& \quad + \sum_{\substack{\carmult^n = \carmulttriv \\ \carmult \neq \carmulttriv}}{\frac{1}{q-1}\sum_{\carmultbis \in \Fqnonzerocar}{\biggl(\sum_{\substack{x \in \Fq^k \\ Q(x) \neq 0}}{\carmult(Q(x))\carmultbis(x_1 \dots x_l)}\biggr) \carmultbis(\lambda)}}.
\end{align*}
Let us compute $N_{\lambda,\carmult,\carmultbis} = \sum_{Q(x) \neq 0}{\carmult(Q(x))\carmultbis(x_1 \dots x_l)}$. As $\alpha_i$ and $\beta_i$ are $>0$,
\begin{multline*}
N_{\lambda,\carmult,\carmultbis} = \smash{\sum_{\substack{(x_1,\dots,x_k) \in (\Fqnonzero)^k \\ \forall i \leq l-1\text{, } x_i \neq 1 \\ x_l+\dots+x_k \neq 1}} (\carmult^{\alpha_1}\carmultbis)(x_1) \carmult^{\beta_1}(1-x_1) \dots (\carmult^{\alpha_{l-1}}\carmultbis)(x_{l-1})} \\ \shoveright{\carmult^{\beta_{l-1}}(1-x_{l-1}) (\carmult^{\alpha_l}\carmultbis)(x_l) \carmult^{\alpha_{l+1}}(x_{l+1}) \dots \carmult^{\alpha_{k}}(x_k)}\\
\carmult^{\beta_l}(1-x_l-\dots-x_k).
\end{multline*}
We recognize a product of Jacobi sums:
\[N_{\lambda,\carmult,\carmultbis} = \jacobisum(\carmult^{\alpha_1}\carmultbis, \carmult^{\beta_1}) \dots \jacobisum(\carmult^{\alpha_{l-1}}\carmultbis, \carmult^{\beta_{l-1}}) \jacobisum(\carmult^{\alpha_l}\carmultbis, \carmult^{\alpha_{l+1}}, \dots, \carmult^{\alpha_{k}}, \carmult^{\beta_l}).\]
By using \frml{\eqref{formula:jacobi.sums}} \pg{\pageref{formula:jacobi.sums}}, we deduce that
\begingroup\footnotesize\[N_{\lambda,\carmult,\carmultbis} = \frac{1}{q^\nu} \frac{\gausssum{\carmult^{\alpha_1}\carmultbis} \dots \gausssum{\carmult^{\alpha_l}\carmultbis} \gausssum{\carmult^{\beta_1}} \dots \gausssum{\carmult^{\beta_l}} \gausssum{\carmult^{\alpha_{l+1}}} \dots \gausssum{\carmult^{\alpha_{k}}}} {\gausssum{\carmult^{\alpha_1+\beta_1}\carmultbis} \dots \gausssum{\carmult^{\alpha_{l-1}+\beta_{l-1}}\carmultbis} \gausssum{\carmult^{\alpha_l+\dots+\alpha_{k}+\beta_l}\carmultbis}},\]\endgroup
with $\nu$ as defined in the theorem.
\end{proof}

\begin{notations}
Let $N_{\lambda,\carmult,\carmultbis}$ be as in the previous theorem; we define
\[N_{\lambda,\carmult} = \frac{1}{q-1} \sum_{\carmultbis \in \Fqnonzerocar}{\mathinner{N_{\lambda,\carmult,\carmultbis}} \carmultbis(\lambda)} \quad \text{and} \quad N_{\lambda} = \sum_{\substack{\carmult^n = \carmulttriv \\ \carmult \neq \carmulttriv}}{N_{\lambda,\carmult,\carmultbis}}.\]
\end{notations}

\begin{corollary}\label{result:nb.zeros:hypergeom:reorganise}
Assume that $n$ is odd, that none of the elements of the sequence $(\beta_1,\dots,\beta_{l},\allowbreak\alpha_{l+1},\dots,\alpha_k)$ are divisible by $n$ and that, for $1 \leq b \leq n-1$, the number of terms of the sequence $\equiv b \mod n$ is equal to the number of terms $\equiv -b \mod n$ (this implies that $k$ is even). When these conditions are met, we say we have \emph{complete pairing}. In this case,
\[N_{\lambda,\carmult,\carmultbis} = q^{\frac{k}{2}-\nu} \frac{\gausssum{\carmult^{\alpha_1}\carmultbis} \dots \gausssum{\carmult^{\alpha_l}\carmultbis}} {\gausssum{\carmult^{\alpha_1+\beta_1}\carmultbis} \dots \gausssum{\carmult^{\alpha_{l-1}+\beta_{l-1}}\carmultbis} \gausssum{\carmult^{\alpha_l+\dots+\alpha_k+\beta_l}\carmultbis}},\]
where $\nu$ is the number of trivial characters appearing in the denominator.
\end{corollary}

\begin{proof}
This is an immediate consequence of the reflection formula~\eqref{formula:reflection}:
\[\gausssum{\carmult^{\beta_1}} \dots \gausssum{\carmult^{\beta_l}} \gausssum{\carmult^{\alpha_{l+1}}} \dots \gausssum{\carmult^{\alpha_k}} = q^{\frac{k}{2}}.\]
(Let us note that, because $\carmult \neq \carmulttriv$ and because each $\alpha_i$ and $\beta_j$ are $\not\equiv 0 \mod n$, the characters appearing are all non trivial, and so the reflection formula applies with $\carmult(-1) = 1$ as $n$ is odd.)
\end{proof}

\subsection{Link with some hypergeometric hypersurfaces}\label{subsection:nb.points.hypergeom:hypersurface.hypergeom}

Assume that $n$ is odd and that $\alpha_1 + \beta_1 \equiv 0 \mod n$. In that case, $\hypergeomvar$ has the same number of points as the hypersurface of $\varaff{k}$ defined by
\begin{multline*}
y^n = x_2^{\alpha_2} \dots x_k^{\alpha_k} (1-x_2)^{\beta_2} \dots (1-x_{l-1})^{\beta_{l-1}}\\ \cdot (1-x_l-\dots-x_k)^{\beta_{l+1}} (1-\lambda x_2 \dots x_l)^{\beta_1}
\end{multline*}
without the points where $x_2 \dots x_l = 0$. We recover in this way a hypersurface of the same type as in \cite[\subsctn{11.1}]{CdlORV.II} when $n=5$ (see also example~\ref{example:n=5} page~\pageref{example:n=5}).

\section{Number of points of the dwork hypersurfaces}\label{section:nb.points.dwork}

In all this \sctn{\ref{section:nb.points.dwork}}, $n$ denotes an integer $\geq 3$ and our aim is to compute the number of points of $\dworkhypersurface$ and then organise it into an appropriate form to relate it to the number of points of varieties of hypergeometric type considered in \sctn{\ref{section:nb.points.hypergeom}}.

To compute the number of points of $\dworkhypersurface$ in terms of Gauss sums, it is possible to use a method close to the one A.~Weil used in \cite{Weil.nb.sols.eqn.ffields} for the diagonal case $\parameter = 0$; this is done for example in \cite[\thrm{2}, \pg{13}]{Koblitz.nb.points} and \cite[\sctn{3}]{Wan.mirror}. After recalling this computation in \subsctn{\ref{subsection:nb.points.dwork:formulas}}, we will organise the terms in the same way as Candelas, de la Ossa and Rodriguez-Villegas did for the case $n=5$ in \cite[\sctn{9}]{CdlORV.I} and \cite[\sctn{11}]{CdlORV.II}, namely (see \thrm{\ref{result:nb.zeros:psi<>0.organise}}):
\[\card{\dworkhypersurface(\Fq)} = 1 + q + \dots + q^{n-2} + N_\textup{mirror} + \sum{N_{s}}.\]
In \sctn{\ref{section:nb.points.link}}, we will explain how each $N_{s}$ is related to a $N_{\lambda} = \card{H_{\lambda}(\Fq)} - (q-1)^{l-1}q^{k-l}$ from~\sctn{\ref{section:nb.points.hypergeom}} (here, $\lambda = \frac{1}{\parameter^n}$).

\subsection{Preliminaries}\label{subsection:nb.points.dwork:prelim}

The aim of this~\subsctn{\ref{subsection:nb.points.dwork:prelim}} is to set a certain number of notations useful in what follows. The groups $\ZnZ$, $\ZnZinv$ and $\Sn$ act on each $(s_1,\dots,s_n) \in (\ZnZ)^n$ satisfying $s_1 + \dots + s_n = 0$ in the following way:
\begin{alignat*}{2}
& \forall j \in \ZnZ, \quad&& j \cdot (s_1,\dots,s_n) = (s_1+j,\dots,s_n+j);\\
& \forall k \in \ZnZinv, \quad&& k \times (s_1,\dots,s_n) = (ks_1,\dots,ks_n);\\
& \forall \sigma \in \Sn, \quad&& {}^\sigma (s_1,\dots,s_n) = (s_{\sigma^{-1}(1)},\dots,s_{\sigma^{-1}(n)}).
\end{alignat*}

\begin{definition}\label{definition:gammas}
Consider an $s = (s_1,\dots,s_n) \in (\ZnZ)^n$ such that $s_1 + \dots + s_n = 0$; we denote by
\begin{subdefinition}
    \item $\classZnZ{s} = \classZnZ{s_1,\dots,s_n}$ the class of $(s_1,\dots,s_n)$ mod the action of $\ZnZ$;
    \item $\classSn{s} = \classSn{s_1,\dots,s_n}$ the class of $(s_1,\dots,s_n)$ mod the simultaneous actions of $\ZnZ$ and $\Sn$;
    \item $\classSnZnZinv{s}$ the class of $(s_1,\dots,s_n)$ mod the simultaneous actions of $\ZnZ$, $\Sn$ and $\ZnZinv$;
    \item $\multSn{s}$ the number of permutations of $(s_1,\dots,s_n)$.
\end{subdefinition}
\end{definition}

\begin{remarks} \ 
\begin{subremarks}
    \item The number $\multSn{s}$ only depends on $\classSnZnZinv{s}$, not on the choice of $s$.
    \item If all the $s_i$ are equal, then $\multSn{s} = 1$.
    \item If $\classSn{s} = \classSn{0,1,2,\dots,n-1}$, then $\multSn{s} = n!$ but the number of permutations of $\classZnZ{s}$ is $n!/n = (n-1)!$ (the $1/n$ comes from the fact that adding the same number to each coordinate amounts to a circular permutation).
\end{subremarks}
\end{remarks}

The following lemma, which we will only use later (see \lmm{\ref{result:K.divides.gamma}}), shows that, when $n$ is prime, the number $\multSn{s}$ of permutations of $(s_1,\dots,\allowbreak s_n)$ is almost always the same as the number of permutations of $\classZnZ{s_1,\dots,s_n}$.

\begin{lemma}\label{result:nb.perm.s=nb.perm.[s]}
Assume that $n$ is prime. If $\classSn{s_1,\dots,s_n} \neq \classSn{0,1,2,\dots,n-1}$, then $\multSn{s}$ is equal to the number of permutations of $\classZnZ{s_1,\dots,s_n}$.
\end{lemma}

\begin{proof}
If there exists $j \in \ZnZ$ non zero such that $(s_1+j,\dots,s_n+j)$ is a permutation of $(s_1,\dots,s_n)$, then $\set{s_1,\dots,s_n}$ is a nonempty subset of $\ZnZ$ stable by $x \mapsto x+j$ and thus equal to $\ZnZ$ as $n$ is prime. Consequently, $\classSn{s} = \classSn{0,1,2,\dots,n-1}$.
\end{proof}

\begin{remark}\label{remarque:existence.j}
This proof shows that, when $\classSn{s} \neq \classSn{0,1,\dots,n-1}$, the only $j \in \ZnZ$ such that there exists $\sigma \in \Sn$ satisfying ${}^{\sigma\!}s = s + j$ is $j=0$.
\end{remark}

\subsection{Formula for the number of points of \texorpdfstring{$\dworkhypersurface$}{Xpsi}}\label{subsection:nb.points.dwork:formulas}

The aim of this~\subsctn{\ref{subsection:nb.points.dwork:formulas}} is to prove~\thrm{\ref{result:nb.zeros:psi<>0}} below, stated in a slightly different form by Koblitz in \cite[\sctn{3}]{Koblitz.nb.points}. From now on, we resume using the notations and assumptions of the introduction: $\Fq$ is a finite field, $n$ an integer $\geq 3$ such that $q \equiv 1 \mod n$, $\parameter \in \Fq$ is a non-zero parameter (but we don't yet suppose that $\parameter^n \neq 1$) and $\dworkhypersurface$ is the hypersurface of $\varprojFq{n-1}$ given by $x_1^n + \dots + x_n^n - n \parameter x_1 \dots x_n = 0$. 

\begin{theorem}[Koblitz]\label{result:nb.zeros:psi<>0}
We have
\begin{align*}
\card{\dworkhypersurface(\Fq)}
& = 1 + q + \dots + q^{n-2} \\
& \quad + \frac{1}{q-1} \sum_{\vphantom{\Fqnonzerocar}\classZnZ{s}}{\sum_{\carmultbis \in \Fqnonzerocar} \frac{1}{q^\delta}{\biggl({\prod_{i=1}^{n}\gausssum{\carmult^{-s_i}\carmultbis^{-1}}}\biggl) \gausssum{\carmultbis^n} \carmultbis(\tfrac{1}{(-n\parameter)^n})}},
\end{align*}
where $\delta = 0$ if one of the $\carmult^{s_i}\carmultbis$ is trivial and $\delta = 1$ otherwise.
\end{theorem}

\begin{proof}
For the sake of completeness, and because it would be just as long to deduce our formula from Koblitz', we will recall the proof given in \cite[\sctn{3}]{Koblitz.nb.points}.

Let $f(x) = x_1^n + \dots + x_n^n - n \parameter x_1 \dots x_n$ and set
\begin{align*}
& \nbzerosaffdworkhypersurface = \card{\setst{x \in \Fq^n}{f(x) = 0}}; \\
& \nbzerosaffdworkhypersurfaceallnotzero = \card{\setst{x \in (\Fqnonzero)^n}{f(x) = 0}}.
\end{align*}
As the product $x_1 \dots x_n$ is zero when one of the $x_i$ is zero, we have $\nbzerosaffdworkhypersurface - \nbzerosaffdworkhypersurfaceallnotzero = \nbzerosaffdworkhypersurface[0] - \nbzerosaffdworkhypersurfaceallnotzero[0]$, i.e.
\[\nbzerosaffdworkhypersurface = \nbzerosaffdworkhypersurface[0] + \nbzerosaffdworkhypersurfaceallnotzero - \nbzerosaffdworkhypersurfaceallnotzero[0].\]
The computation of $\nbzerosaffdworkhypersurface[0]$ is classical and goes back to A.~Weil, also we will not recall it (see \cite{Weil.nb.sols.eqn.ffields} or \cite[\thrm{10.4.2}, \pg{304}]{BEW.gauss.jacobi.sums}). By using \frml{\eqref{formula:jacobi.sums}} to express everything in terms of Gauss sums and by doing the change of variable $\carmult_i \mapsto \carmult_i^{-1}$, here is what we find:
\begin{equation}\label{equation:nbzerosaffdworkhypersurface.0}
\nbzerosaffdworkhypersurface[0] = q^{n-1} + \frac{q-1}{q} \sum_{\substack{\carmult_i^n = \carmulttriv\text{, }\carmult_i \neq \carmulttriv \\ \carmult_1\dots\carmult_n = \carmulttriv}}{\biggl(\prod_{i=1}^{n}{\gausssum{\carmult_i^{-1}}}\biggl)}.
\end{equation}
We now need to compute $\nbzerosaffdworkhypersurfaceallnotzero$ and $\nbzerosaffdworkhypersurfaceallnotzero[0]$. Both computations rely on the same method, the only difference being that, when $\parameter = 0$, the polynomial $f(x)$ is a sum of $n$ monomials instead of $n+1$ which slightly changes the result. We will only give the details for $\nbzerosaffdworkhypersurfaceallnotzero$ when $\parameter \neq 0$.

The orthogonality formula~\eqref{formula:orthogonality:car.add} \pg{\pageref{formula:orthogonality:car.add}} for additive characters shows that
\begin{align*}
\nbzerosaffdworkhypersurfaceallnotzero
& = \frac{1}{q} \sum_{a \in \Fq}{\sum_{x \in (\Fqnonzero)^n}}{\caradd(af(x))} \\
& = \frac{(q-1)^n}{q} + \frac{1}{q} \sum_{a \in \Fqnonzero}{\sum_{x \in (\Fqnonzero)^n}}{\biggl(\prod_{i=1}^{n}{\caradd(ax_i^n)}\biggl)\caradd(-n\parameter a x_1\dots x_n)}.
\end{align*}
We now express each $\caradd(\ldots)$ in terms of Gauss sums thanks to \frml{\eqref{formula:caradd=gauss.sums}} \pg{\pageref{formula:caradd=gauss.sums}}:
\begin{multline*}
\nbzerosaffdworkhypersurfaceallnotzero = \frac{(q-1)^n}{q} \\
\shoveright{+ \frac{1}{q}\sum_{\carmultbis_1,\dots,\carmultbis_{n+1} \in \Fqnonzerocar}\biggl(\prod_{i=1}^{n+1}{\gausssum{\carmultbis_i^{-1}}}\biggl) \biggl(\frac{1}{q-1}\sum_{a \in \Fqnonzero}{(\carmultbis_1\dots\carmultbis_{n+1})(a)}\biggl)} \\ \prod_{i=1}^{n}{\biggl(\frac{1}{q-1}\sum_{x_i \in \Fqnonzero}{(\carmultbis_i^n\carmultbis_{n+1})(x_i)}\biggl)}\carmultbis_{n+1}(-n\parameter).\end{multline*}
Using orthogonality formulas, the sums over $a$ and the $x_i$ are all non-zero (equal to $q-1$) if and only if
\[\begin{system}\carmultbis_1\dots\carmultbis_n\carmultbis_{n+1} = \carmulttriv \\ \forall i \in \integerinterval{1}{n}, \quad \carmultbis_i^n \carmultbis_{n+1} = \carmulttriv \end{system} \text{i.e.} \quad \exists \carmultbis \in \Fqnonzerocar, \quad \begin{system} \carmultbis_i = \carmult_i \carmultbis \\ \carmult_i^n = \carmulttriv \text{ and } \carmult_1\dots\carmult_n = \carmulttriv \\ \carmultbis_{n+1} = \carmultbis^{-n}\end{system}\]
The character $\carmultbis$ defined in this way is not unique; indeed, if $\carmultbis '$ and $\carmult_i '$ are also solutions of the system, there exists $\carmult$ satisfying $\carmult^n = \carmulttriv$ such that $\carmultbis ' = \carmult^{-1} \carmultbis$ and $\carmult_i ' = \carmult \carmult_i$ for all $i$. This means that if $R$ is a representative set of the $n$-uples $(\carmult_1,\dots,\carmult_n)$ of characters mod the $(\carmult,\dots,\carmult)$ satisfying $\carmult_i^n = \carmulttriv$ and $\carmult_1\dots \carmult_n = \carmulttriv$ with $\carmult^n = \carmulttriv$, the map $(\carmult_1,\dots,\carmult_n,\carmultbis) \mapsto (\carmult_1\carmultbis,\dots,\carmult_n\carmultbis,\carmultbis^{-n})$ is a one-to-one map of $R \times \Fqnonzerocar$ onto the set of $(n+1)$-uples $(\carmultbis_1,\dots,\carmultbis_{n+1})$ satisfying the preceding conditions. From this, it results that, if $\carmult$ is a multiplicative character of order $n$,
\begin{align}\label{equation:nbzerosaffdworkhypersurfaceallnotzero}
\nbzerosaffdworkhypersurfaceallnotzero
& = \frac{(q-1)^n}{q} \\
\notag & \quad + \frac{1}{q}\sum_{\substack{\vphantom{\Fqnonzerocar}\classZnZ{s}}}{\sum_{\carmultbis \in \Fqnonzerocar} {\biggl(\prod_{i=1}^{n}{\gausssum{\carmult^{-s_i}\carmultbis^{-1}}}\biggl) \gausssum{\carmultbis^n} \carmultbis(\tfrac{1}{(-n\parameter)^n})}}.
\end{align}
This ends the computation of $\nbzerosaffdworkhypersurfaceallnotzero$. By a similar method, we find
\begin{equation}\label{equation:nbzerosaffdworkhypersurfaceallnotzero.0}
\nbzerosaffdworkhypersurfaceallnotzero[0] = \frac{(q-1)^n}{q} + \frac{q-1}{q}\sum_{\substack{\carmult_i^n = \carmulttriv \\ \carmult_1\dots\carmult_n = \carmulttriv}} {\biggl(\prod_{i=1}^{n}{\gausssum{\carmult_i^{-1}}}\biggl)}.
\end{equation}

From \eqref{equation:nbzerosaffdworkhypersurface.0} and \eqref{equation:nbzerosaffdworkhypersurfaceallnotzero.0}, we obtain
\begin{align*}
\nbzerosaffdworkhypersurface[0] - \nbzerosaffdworkhypersurfaceallnotzero[0]
& = q^{n-1} - \frac{(q-1)^n}{q} - \frac{q-1}{q} \sum_{\substack{\carmult_i^n = \carmulttriv \\ \carmult_1\dots\carmult_n = \carmulttriv \\ \exists i\text{, } \carmult_i = \carmulttriv}} {\biggl(\prod_{i=1}^{n}{\gausssum{\carmult_i^{-1}}}\biggl)}, \\
& = q^{n-1} - \frac{(q-1)^n}{q} \\
& \quad - \frac{q-1}{q} \sum_{\substack{(\carmult_1,\dots,\carmult_n) \bmod \set{(\carmult,\dots,\carmult)} \\ \carmult_i^n = \carmulttriv,\text{ } \carmult_1\dots\carmult_n = \carmulttriv \\ \exists i\text{, } \carmult_i = \carmulttriv}}{\sum_{\substack{\carmultbis \in \Fqnonzerocar\\\carmultbis^n=\carmulttriv}} {\biggl(\prod_{i=1}^{n}{\gausssum{(\carmult_i\carmultbis)^{-1}}}\biggl)}},
\end{align*}
Writing $\carmult_i = \carmult^{s_i}$ where $\carmult$ is, as above, a character of order $n$, we transform the first sum into a sum over the $\classZnZ{s}$ such that $\exists i$, $s_i = 0$; finally, we combine the terms of this sum with those satisfying $\carmultbis^n = \carmulttriv$ in \frml{\eqref{equation:nbzerosaffdworkhypersurfaceallnotzero}} above for $\nbzerosaffdworkhypersurfaceallnotzero$. As $\gausssum{\carmulttriv} = -1$, we have, with $\delta$ as defined in the theorem,
\begin{align*}
\nbzerosaffdworkhypersurface
& = \nbzerosaffdworkhypersurfaceallnotzero + \nbzerosaffdworkhypersurface[0] - \nbzerosaffdworkhypersurfaceallnotzero[0] \\
& = q^{n-1} + \sum_{\vphantom{\Fqnonzerocar}\classZnZ{s}}{\sum_{\carmultbis \in \Fqnonzerocar} {\frac{1}{q^\delta}\biggl(\prod_{i=1}^{n}{\gausssum{\carmult^{-s_i}\carmultbis^{-1}}}\biggl) \gausssum{\carmultbis^n} \carmultbis(\tfrac{1}{(-n\parameter)^n})}}.
\end{align*}
By counting the number of zeros in the projective space instead of the affine space, we obtain the announced formula.
\end{proof}

\subsection{Reorganisation of the terms}\label{subsection:nb.points.dwork:reorg}

We keep the assumptions and notations of \subsctn{\ref{subsection:nb.points.dwork:formulas}} and suppose that $n$ is odd. The aim of this~\subsctn{\ref{subsection:nb.points.dwork:reorg}} is to write the formula obtained for $\card{\dworkhypersurface(\Fq)}$ in \thrm{\ref{result:nb.zeros:psi<>0}} in terms of some coefficients $\coeffnbzeros{s_1,\dots,s_n}{\carmult}{\carmultbis}$ which we now define.

\begin{definition}\label{definition:beta}
Let us consider $(s_1,\dots,s_n) \in (\ZnZ)^n$ such that $s_1 + \dots + s_n = 0$. If $\carmult$ is a multiplicative character of $\Fqnonzero$ of order $n$ and if $\carmultbis$ is a character of $\Fqnonzero$, we set
\begin{equation}\label{def:beta}
\coeffnbzeros{s_1,\dots,s_n}{\carmult}{\carmultbis} = q^{\frac{n+1}{2}-z-\delta} \mathinner{\frac{\gausssum{\carmultbis}\gausssum{\carmult\carmultbis} \dots \gausssum{\carmult^{n-1}\carmultbis}}{\gausssum{\carmult^{s_1}\carmultbis} \dots \gausssum{\carmult^{s_n}\carmultbis}}},
\end{equation}
where $z$ denotes the number of trivial characters in the finite sequence $(\carmult^{s_1}\carmultbis,\dots,\carmult^{s_n}\carmultbis)$ and where $\delta = 0$ if $z \neq 0$ and $\delta = 1$ if $z = 0$ (this is the same $\delta$ as in \thrm{\ref{result:nb.zeros:psi<>0}}).
\end{definition}

\begin{proposition}
With the above assumptions, we have
\[\frac{1}{q^\delta}\biggl(\prod_{i=1}^{n}{\gausssum{\carmult^{-s_i}\carmultbis^{-1}}} \biggr)\gausssum{\carmultbis^n}\carmultbis(\tfrac{1}{(-n\parameter)^n}) = \mathinner{\coeffnbzeros{s_1,\dots,s_n}{\carmult}{\carmultbis}} \carmultbis(\tfrac{1}{\parameter^n}).\]
\end{proposition}

\begin{proof}
Invoking the reflection formula~\eqref{formula:reflection}, we obtain
\[\prod_{i=1}^{n}{\gausssum{\carmult^{-s_i}\carmultbis^{-1}}} = q^{n-z} \mathinner{\frac{\carmultbis(-1)^n}{\gausssum{\carmult^{s_1}\carmultbis} \dots \gausssum{\carmult^{s_n}\carmultbis}}},\]
and, using the multiplication formula~\eqref{formula:multiplication}, we get, as $n$ is odd,
\[\gausssum{\carmultbis^n} = \mathinner{\frac{\carmultbis(n)^n}{q^{\frac{n-1}{2}}}} \gausssum{\carmultbis} \gausssum{\carmult\carmultbis} \dots \gausssum{\carmult^{n-1}\carmultbis}.\]
With these two formulas, we deduce at once the result.
\end{proof}

The coefficients $\coeffnbzerosempty$ defined above satisfy the following three compatibility relations respective to the actions of the groups $\ZnZ$, $\Sn$ and $\ZnZinv$.

\begin{lemma}\label{result:beta:propr}
With the same notations and assumptions as the preceding definition,
\begin{alignat}{2}
&\label{formula:beta.Sn}\forall \sigma \in \Sn,\quad && \beta_{(s_{\sigma(1)},\dots,s_{\sigma(n)}),\carmult,\carmultbis} = \beta_{(s_1,\dots,s_n),\carmult,\carmultbis}; \\
&\label{formula:beta.ZnZ}\forall j \in \Z,\quad && \beta_{(s_1+j,\dots,s_n+j),\carmult,\carmultbis} = \beta_{(s_1,\dots,s_n),\carmult,\carmult^j\carmultbis}; \\
&\label{formula:beta.ZnZinv}\forall k \in \ZnZinv,\quad && \beta_{(ks_1,\dots,ks_n),\carmult,\carmultbis} = \beta_{(s_1,\dots,s_n),\carmult^k,\carmultbis}.
\end{alignat}
\end{lemma}

\begin{proof}
\frml{\eqref{formula:beta.Sn}} results immediately from the definition of $\coeffnbzerosempty$. As for \eqref{formula:beta.ZnZ} and \eqref{formula:beta.ZnZinv}, we note that the product $\gausssum{\carmultbis}\gausssum{\carmult\carmultbis} \dots \allowbreak\gausssum{\carmult^{n-1}\carmultbis}$ in \frml{\eqref{def:beta}} stays the same if we change $\carmultbis$ into $\carmult^j \carmultbis$ or if we change $\carmult$ into $\carmult^k$ with $k$ prime to $n$.
\end{proof}

\begin{proposition}\label{definition:Ns}
Under the same assumptions as above, the following quantities only depend on $\classSn{s}$ (as well as on the choice of $\carmult$) and of $\classSnZnZinv{s}$ respectively and not on the choice of the representative $(s_1,\dots,s_n)$:
\begin{align*}
& N_{\classSn{s},\carmult} = \frac{1}{q-1} \sum_{\carmultbis \in \Fqnonzerocar}{\mathinner{\coeffnbzeros{s_1,\dots,s_n}{\carmult}{\carmultbis}} \carmultbis(\tfrac{1}{\parameter^n})};\\
& N_{\classSnZnZinv{s}} = \multSn{s} \sum_{\classSn{s'} \in \classSnZnZinv{s}}{N_{\classSn{s'},\carmult}}.
\end{align*}
\end{proposition}

\begin{proof}
For $N_{\classZnZ{s},\carmult}$, we just use \frml{\eqref{formula:beta.ZnZ}} and the fact that $\carmultbis \mapsto \carmult^j \carmultbis$ is a one-to-one map of $\Fqnonzerocar$ onto itself when $j \in \ZnZ$. For $N_{\classSnZnZinv{s}}$, we use \frml{\eqref{formula:beta.ZnZinv}} and the fact that $\carmult \mapsto \carmult^k$ is a one-to-one map of $\setst{\carmult \in \Fqnonzerocar}{\carmult^n = \carmulttriv}$ onto itself if $k \in \ZnZinv$.
\end{proof}

We deduce the following result.

\begin{theorem}\label{result:nb.zeros:psi<>0.organise}
Under the preceding assumptions, we have
\[\card{\dworkhypersurface(\Fq)} = 1 + q + \dots + q^{n-2} + \sum_{\classSnZnZinv{s}}{N_{\classSnZnZinv{s}}}.\]
\end{theorem}

\begin{remark}
As we will see in \subsctn{\ref{section:nb.points.dwork:identification}} below, $N_{\classSnZnZinv{0}} = N_{\textup{mirror}}$ and, when $\dworkhypersurface$ is non-singular (i.e. when $\parameter^n \neq 1$), $N_{\classSnZnZinv{(0,1,2,\dots,n-1)}} = 0$.
\end{remark}

\subsection{Identification of some of the factors}\label{section:nb.points.dwork:identification}

We keep the assumptions and notations of \subsctn{\ref{subsection:nb.points.dwork:reorg}}. Let us recall that $\mirrordworkhypersurface$ denotes the ``singular mirror'' of $\dworkhypersurface$, as specified in the introduction, and we write $N_{\textup{mirror}} = \card{\mirrordworkhypersurface(\Fq)}-(1+q+\dots+q^{n-2})$.

\begin{theorem}[Wan]\label{result:nb.zeros.mirror.factor}
$N_{\classSnZnZinv{0}} = N_{\textup{mirror}}$.
\end{theorem}

\begin{proof}
See \cite[\sctn{4}]{Wan.mirror}; note that the result is not known when $q \not\equiv 1 \mod n$, unless $n$ is prime (see \cite{Haessig.(n+1)-adic}).
\end{proof}

Let us recall that, in this \sctn{\ref{section:nb.points.dwork}}, the only assumption we make on $\parameter$ is that $\parameter \neq 0$.

\begin{lemma}\label{resultat:nb.zeros.terme.sing}
We have
\[N_{\classSn{0,1,2,\dots,n-1},\carmult} = \begin{cases} 0 & \text{if $\parameter^n \neq 1$,} \\ q^{\frac{n-1}{2}} & \text{if $\parameter^n = 1$,}\end{cases}\]
and so the term $N_{\classSnZnZinv{(0,1,2,\dots,n-1)}} = \mathinner{(n-1)!} N_{\classSn{0,1,2,\dots,n-1},\carmult}$ does not contribute to the zeta function $\funzeta{\dworkhypersurface/\Fq}{t}$ when $\parameter^n \neq 1$ and contributes as $(1-q^{\frac{n-1}{2}}t)^{-(n-1)!}$ when $\parameter^n = 1$.
\end{lemma}

\begin{proof}
When $\classSn{s_1,\dots,s_n} = \classSn{0,1,\dots,n-1}$, we have
\[\gausssum{\carmult^{s_1}\carmultbis} \dots \gausssum{\carmult^{s_{n}}\carmultbis} = \gausssum{\carmultbis} \gausssum{\carmult\carmultbis} \dots \gausssum{\carmult^{n-1}\carmultbis}.\]
Moreover, the number $z$ of trivial characters in the sequence $(\carmultbis,\carmult\carmultbis,\dots,\allowbreak\carmult^{n-1}\carmultbis)$ is equal to $1-\delta$ with the notations of \dfntn{\ref{definition:beta}} \pg{\pageref{definition:beta}}, and thus
\[\coeffnbzeros{0,1,\dots,n-1}{\carmult}{\carmultbis} = q^{\frac{n-1}{2}}.\]
Consequently,
\[N_{\classSn{0,1,2,\dots,n-2,n-1},\carmult} = \frac{q^{\frac{n-1}{2}}}{q-1} \sum_{\carmultbis \in \Fqnonzerocar}{\carmultbis(\tfrac{1}{\parameter^n})},\]
and we conclude by using an orthogonality formula.
\end{proof}

\begin{remark}
A similar result was given by Candelas, de la Ossa and Rodriguez-Villegas when $q = p$ and $n=5$ (see \cite[\sctn{9.3}]{CdlORV.I}).
\end{remark}

\section{Link between the number of points}\label{section:nb.points.link}

In all this \sctn{\ref{section:nb.points.link}}, we will assume that the integer $n$ is a prime $\geq 5$ and that $q \equiv 1 \mod n$. We will only add the assumption that $\parameter^n = 1$ in \thrm{\ref{result:conclusion}}.

The aim of this section is to show (in \subsctn{\ref{subsection:nb.points.link:conclusion}}) the \frml{\eqref{formula:AIM}} of the introduction. More precisely, we shall show, in \thrm{\ref{result:link.nb.points}}, that each $N_{\classSnZnZinv{s}}$ (with $\classSnZnZinv{s} \neq \classSnZnZinv{0}$)\footnote{Let us note that there does not exist any $\classSnZnZinv{s} \neq \classSnZnZinv{0}$ when $n=3$; this explains the assumption that $n \geq 5$.} appearing in \thrm{\ref{result:nb.zeros:psi<>0.organise}} is equal, up to a multiplicative integer constant and a power of $q$, to a term of the form
\begin{equation}\label{formula:N.hyp}
N_{\lambda} = \sum_{\substack{\carmult^n = \carmulttriv \\ \carmult \neq \carmulttriv}}N_{\lambda,\carmult} = \sum_{\substack{\carmult^n = \carmulttriv \\ \carmult \neq \carmulttriv}}\frac{1}{q-1}\sum_{\carmultbis \in \Fqnonzerocar}{N_{\lambda,\carmult,\carmultbis} \mathinner{\carmultbis(\lambda)}},
\end{equation}
where $\lambda = \frac{1}{\parameter^n}$ and $N_{\lambda,\carmult,\carmultbis}$ is given by \crllr{\ref{result:nb.zeros:hypergeom:reorganise}} \pg{\pageref{result:nb.zeros:hypergeom:reorganise}}.

The crucial point is, starting from a given $\classSnZnZinv{s}$, to find the integers $\alpha_i$ and $\beta_j$ which appear. For that, we define in \subsctn{\ref{subsection:nb.points.link:prelim}} integers $v_i$ and $w_i$ from which we then define the integers $\alpha_i$ and $\beta_j$ in \subsctn{\ref{subsection:nb.points.link:link}}. But before this, we start by a divisibility result useful for the main result.

\subsection{A divisibility result}\label{subsection:nb.points.link:K|gamma}

The aim of this \subsctn{\ref{subsection:nb.points.link:K|gamma}} is to show that the integer $\multSn{s}$ (from \dfntn{\ref{definition:gammas}} \pg{\pageref{definition:gammas}}) is divisible by
\[\multZnZinv{s} = \card{\setst{k \in \ZnZinv}{\classZnZ{ks_1,\dots,ks_n} \text{ is a permutation of $\classZnZ{s_1,\dots,s_n}$}}}.\]
This result is crucial in \thrm{\ref{result:link.nb.points}} to be sure that the quotient $\multSn{s}/\multZnZinv{s}$ is an integer. Note that $\multZnZinv{s}$ only depends on $\classSnZnZinv{s}$, not on the choice of $s$.

\begin{definition}
Given $s \in (\ZnZ)^n$ such that $s_1+\dots+s_n=0$, we consider the following subgroups of $\Sn$:
\begin{align*}
& \SaQellbarprime = \setst{\sigma \in \Sn}{{}^{\sigma\!}s = s};\\
& \SaQellbar = \setst{\sigma \in \Sn}{\classZnZ{{}^{\sigma\!}s} = \classZnZ{s}}; \\
& \SaQell = \setst{\sigma \in \Sn}{\classZnZ{{}^{\sigma\!}s} \in \ZnZinv \cdot \classZnZ{s}}.
\end{align*}
\end{definition}

Let us note that, with these notations, $[\Sn:\SaQellbarprime]$ is the number $\multSn{s}$ of permutations of $(s_1,\dots,s_n)$ whereas $[\Sn:\SaQellbar]$ is the number of permutations of $\classZnZ{s_1,\dots,s_n}$.

\begin{lemma}\label{result:K.divides.gamma}
When $\classSnZnZinv{s} \neq \classSnZnZinv{0}$, the integer $\multZnZinv{s}$ divides $[\Sn:\SaQellbar]$. Hence, when additionally $\classSn{s} \neq \classSn{0,1,2,\dots,\allowbreak n-1}$, $\multZnZinv{s}$ divides $\multSn{s} = [\Sn:\SaQellbarprime] = [\Sn:\SaQellbar]$.
\end{lemma}

\begin{proof}
We remark that
\[\multZnZinv{s} = \frac{\card{\SaQell}}{\card{\SaQellbar}} \cdot \card{\setst{k \in \ZnZinv}{\classZnZ{k s} = \classZnZ{s}}}.\]
As $\classZnZ{s} \neq \classZnZ{0,\dots,0}$, we have $\card{\setst{k \in \ZnZinv}{\classZnZ{k s} = \classZnZ{s}}} = 1$ and so
\[[\Sn:\SaQellbar] = [\Sn:\SaQell] \cdot \multZnZinv{s}.\]
When furthermore $\classSn{s} \neq \classSn{0,1,2,\dots,\allowbreak n-1}$, we have $\multSn{s} = [\Sn:\SaQellbar]$ by \lmm{\ref{result:nb.perm.s=nb.perm.[s]}} \pg{\pageref{result:nb.perm.s=nb.perm.[s]}}, hence the result.
\end{proof}

\subsection{Transformation of the \texorpdfstring{$\coeffnbzerosempty$}{beta} coefficients}\label{subsection:nb.points.link:prelim}

In order to relate $N_{\classSnZnZinv{s}}$ to a certain $N_{1/\parameter^n}$ in \subsctn{\ref{subsection:nb.points.link:link}}, we must first change the formula giving $\coeffnbzeros{s_1,\dots,s_n}{\carmult}{\carmultbis}$.

\begin{notations}
Consider $(s_1,\dots,s_n) \in (\ZnZ)^n$ such that $s_1+\dots+s_n = 0$. For each $b \in \ZnZ$, define $k(b) = \card{\setst{i}{s_i=b}}$. We have
\[\sum_{b \in \ZnZ}{k(b)b} = 0 \quad\text{and}\quad \sum_{b \in \ZnZ}{k(b)} = n.\]
We also set $n' = \card{\setst{b \in \ZnZ}{k(b) \neq 0}}$ and $m = n -n'$.
\end{notations}

\begin{remarks}\label{remarque:propr.nprime} \ 
\begin{subremarks}
    \item The integer $n'$ satisfies $1 \leq n' \leq n$ and we have $n' = 1$ if and only if $\classZnZ{s} = \classZnZ{0,\dots,0}$ and $n' = n$ if and only if $\classSn{s} = \classSn{0,1,\dots,n-1}$.
    \item As $n$ is prime, the integer $n'$ is $\neq 2$. Indeed, if $k_1b_1+k_2b_2=0$ with $k_1,k_2 \geq 1$ and $k_1+k_2=n$, then $k_1 \not\equiv 0 \mod n$ and $k_1(b_1-b_2)=0$, hence $b_1 = b_2$.
    \item As $n$ is odd, the integer $n'$ is $\neq n-1$. Indeed, let $s_1$, \dots, $s_{n-1}$ be distinct elements of $\ZnZ$ and denote by $s_n$ the element of $\ZnZ$ not appearing in this sequence; as $n$ is odd, we have $s_1+\dots+s_n = 0$, and so $2s_1+\dots+s_{n-1} = s_1-s_n \neq 0$.
    \item\label{remarque:propr.nprime:encadrement.m} Thus, if $\classSn{s} \neq \classSn{0,1,\dots,n-1}$, then $m \geq 2$ and if, moreover, $\classZnZ{s} \neq \classZnZ{0}$, then $2 \leq m \leq n-3$.
\end{subremarks}
\end{remarks}

\begin{theorem}\label{resultat:formule.beta:k(b)}
With the preceding notations,
\begin{equation*}
\beta_{(s_1,\dots,s_n),\carmult,\carmultbis} = q^{\frac{n-1}{2}-\nu} \mathinner{\frac{\,\prod\limits_{b \in \ZnZ\text{, } k(b)=0}{\gausssum{\carmult^{b} \carmultbis}\phantom{^{k(b)-1}}}} {\,\prod\limits_{b \in \ZnZ\text{, } k(b)\neq 0}{\gausssum{\carmult^{b} \carmultbis}^{k(b)-1}}}},
\end{equation*}
where $\nu = 0$ unless there exists $b$ such that $\carmult^{b}\carmultbis = \carmulttriv$ and $k(b) \neq 0$, in which case $\nu = k(b)-1$.
\end{theorem}

\begin{proof}
From the definition of $\coeffnbzeros{s_1,\dots,s_n}{\carmult}{\carmultbis}$ (\dfntn{\ref{definition:beta}} \pg{\pageref{definition:beta}}), we have
\begin{align*}
\beta_{(s_1,\dots,s_n),\carmult,\carmultbis}
& = q^{\frac{n+1}{2}-z-\delta} \mathinner{\frac{\,\prod\limits_{b \in \ZnZ}{\gausssum{\carmult^{b} \carmultbis}\phantom{^{k(b)}}}} {\,\prod\limits_{b \in \ZnZ}{\gausssum{\carmult^{b} \carmultbis}^{k(b)}}}} \\
& = q^{\frac{n+1}{2}-z-\delta} \mathinner{\frac{\,\prod\limits_{k(b)=0}{\gausssum{\carmult^{b} \carmultbis}\phantom{^{k(b)-1}}}} {\,\prod\limits_{k(b)\neq 0}{\gausssum{\carmult^{b} \carmultbis}^{k(b)-1}}}}.
\end{align*}
We now have to show that $z+\delta = 1+\nu$. Recall that $z$ is the number of trivial characters in the finite sequence $(\carmult^{s_1}\carmultbis,\dots,\carmult^{s_n}\carmultbis)$ and that
$\delta = 0$ if $z \neq 0$ and $\delta = 1$ if $z = 0$. When $z=0$, $\delta = 1$ and $\nu = 0$ hence $z+\delta=1+\nu$. When $z \neq 0$, there exists a unique $b \in \ZnZ$ such that $\carmultbis = \carmult^{-b}$; we thus have $z = k(b)$, $\delta = 0$ and $\nu = k(b)-1$, hence $z+\delta = 1+\nu$.
\end{proof}

\begin{remark}\label{definition:v.w}
Let $(v_1,\dots,v_m)$ be an enumeration of the $b \in \ZnZ$ such that $k(b) = 0$ and let $(w_1,\dots,\allowbreak w_m)$ be an enumeration of the $b \in \ZnZ$ such that $k(b) \geq 2$, each repeated with multiplicity $k(b)-1$. The formula of \thrm{\ref{resultat:formule.beta:k(b)}} can be rewritten as
\begin{equation}\label{formule:beta:w}
\beta_{(s_1,\dots,s_n),\carmult,\carmultbis} = q^{\frac{n-1}{2}-\nu} \mathinner{\frac{\gausssum{\carmult^{v_1}\carmultbis} \dots \gausssum{\carmult^{v_m}\carmultbis}}{\gausssum{\carmult^{w_1}\carmultbis} \dots \gausssum{\carmult^{w_{m}}\carmultbis}}},
\end{equation}
where $\nu$ is the number of trivial characters appearing in the denominator.
\end{remark}

\begin{lemma}\label{resultat:sum.v=sum.w}
With the notations of the preceding remark,
\[v_1 + \dots + v_m \equiv w_1 + \dots + w_m \mod n.\]
\end{lemma}

\begin{proof}
This identity can be rewritten as
\[\sum_{k(b)=0}{b} = \sum_{k(b)\geq 1}{(k(b)-1)b} \quad\text{i.e.}\quad \sum_{b}{b} = \sum_{b}{k(b)b}.\]
We conclude by noting that $\sum_{b \in \ZnZ}{k(b)b} = 0$ and that, because $n$ is odd, $\sum_{b \in \ZnZ}{b} = 0$.
\end{proof}

\subsection{Link with the hypergeometric varieties}\label{subsection:nb.points.link:link} We now establish the link between $\dworkhypersurface$ and the varieties of hypergeometric type from \sctn{\ref{section:nb.points.hypergeom}}.

\begin{theorem}\label{result:link.nb.points}
Let $\classSnZnZinv{s}$ be distinct from the class of $(0,1,\dots,n-1)$ and of $(0,\dots,0)$. If $s$ is a representative of $\classSnZnZinv{s}$, assume that there exists two sequences $(v_1,\dots,v_m)$ and $(w_1,\dots,w_m)$ of elements of $\ZnZ$ as in \rmrk{\ref{definition:v.w}} and an even integer $m' \leq m-2$ such that
\[\forall i \in \integerinterval{1}{\tfrac{m'}{2}}, \quad w_{2i-1} - v_{2i-1} \equiv -(w_{2i} - v_{2i}) \mod n.\]
We consider the affine variety $\hypergeomvar[1/\parameter^n]$ of dimension $2m-m'-3$ given by
\[\begin{system}
y^n = x_1^{v_1} \dots x_{m}^{v_m} x_{m+1}^{v_{m'+1}-w_{m'+1}} \dots x_{2m-m'-2}^{v_{m-2}-w_{m-2}} (1-x_1)^{w_1-v_1} \dots \\
\qquad\qquad(1-x_{m-1})^{w_{m-1}-v_{m-1}} (1-x_m-\dots-x_{2m-m'-2})^{v_{m-1}-w_{m-1}} \\
x_1 \dots x_m = \parameter^n 
\end{system}\]
(In this formula, we replace the exponents by their representatives in $\integerinterval{1}{n}$.) It is a variety of the form considered in \crllr{\ref{result:nb.zeros:hypergeom:reorganise}} \pg{\pageref{result:nb.zeros:hypergeom:reorganise}} and we have, using the notations of \sctn{\ref{section:nb.points.hypergeom}},
\[N_{\classSnZnZinv{s}} = \mathinner{\frac{\multSn{s}}{\multZnZinv{s}}} q^{\frac{n+1}{2}-\frac{2m-m'}{2}} N_{1/\parameter^n}\, \qquad \text{where $\multSn{s}/\multZnZinv{s} \in \N$ by \lmm{\ref{result:K.divides.gamma}}.}\]
\end{theorem}

\begin{proof}
As $\classSnZnZinv{s}$ is distinct from the class of $(0,1,\dots,n-1)$, we have $m \geq 2$ (see \rmrk{\ref{remarque:propr.nprime:encadrement.m}} \pg{\pageref{remarque:propr.nprime:encadrement.m}}). The variety we consider is the one introduced in \thrm{\ref{result:nb.points.hyper.var}} \pg{\pageref{result:nb.points.hyper.var}} with $l = m$, $k = 2m - m' - 2$ and
\begin{align*}
& \alpha_1 = v_1{,} \quad \dots{,} \quad \alpha_m = v_m; \\
& \alpha_{m+1} = v_{m'+1}-w_{m'+1}{,} \quad \dots{,} \quad \alpha_{2m-m'-2} = v_{m-2}-w_{m-2}; \\
& \beta_1 = w_1-v_1{,} \quad \dots{,} \quad \beta_{m-1} = w_{m-1}-v_{m-1}{,} \quad \beta_{m} = v_{m-1}-w_{m-1}.
\end{align*}
According to the pairing assumption on the $v_i$ and $w_i$ and to \lmm{\ref{resultat:sum.v=sum.w}}, we have
\[v_{m'+1} + \dots + v_m = w_{m'+1} + \dots + w_m \quad\text{in $\ZnZ$,}\]
and thus, $\alpha_m+\alpha_{m+1}+\dots+\alpha_{2m-m'-2}+\beta_m \equiv w_m \mod n$. Moreover,
\begin{align*}
&\alpha_1+\beta_1 \equiv w_1 \mod{n}{,} \quad \dots{,} \quad \alpha_{m-1}+\beta_{m-1} \equiv w_{m-1} \mod{n}; \\
&\beta_1+\beta_2 \equiv 0 \mod{n}{,} \quad \dots{,} \quad \beta_{m'-1}+\beta_{m'} \equiv 0 \mod{n}; \\
&\alpha_{m+1}+\beta_{m'+1} \equiv 0 \mod{n}{,} \quad \dots{,} \quad \alpha_{2m-m'-2}+\beta_{m-2} \equiv 0 \mod{n}; \\ &\beta_{m-1}+\beta_{m} \equiv 0 \mod{n}.
\end{align*}
The last three lines show that we have complete pairing (in the sense of \crllr{\ref{result:nb.zeros:hypergeom:reorganise}} \pg{\pageref{result:nb.zeros:hypergeom:reorganise}}) of the sequence $(\beta_1,\dots,\beta_{m},\allowbreak\alpha_{m+1},\dots,\alpha_{2m-m'-2})$; these elements are $\not\equiv 0 \mod n$ as $v_i \not\equiv w_i \mod n$, and so
\[N_{1/\parameter^n,\carmult,\carmultbis} = q^{\frac{2m-m'-2}{2}-\nu} \frac{\gausssum{\carmult^{v_1}\carmultbis} \dots \gausssum{\carmult^{v_m}\carmultbis}}{\gausssum{\carmult^{w_1}\carmultbis} \dots \gausssum{\carmult^{w_m}\carmultbis}}.\]
Hence, by comparing with \frml{\eqref{formule:beta:w}} \pg{\pageref{formule:beta:w}},
\[\beta_{(s_1,\dots,s_n),\carmult,\carmultbis} = q^{\frac{n+1}{2}-\frac{2m-m'}{2}} N_{1/\parameter^n,\carmult,\carmultbis}.\]
Multiplying this equality by $\frac{1}{q-1} \carmultbis(\tfrac{1}{\parameter^n})$ and summing over $\carmultbis \in \Fqnonzerocar$, we get
\[N_{\classSn{s},\carmult} = q^{\frac{n+1}{2}-\frac{2m-m'}{2}} N_{1/\parameter^n,\carmult}.\]
We now sum over $k \in \integerinterval{1}{n-1}$ the preceding formula where $\carmult$ is remplaced by $\carmult^k$. Noting that $N_{\classSn{s},\carmult^k} = N_{\classSn{ks},\carmult}$ (see \frml{\eqref{formula:beta.ZnZinv}} \pg{\pageref{formula:beta.ZnZinv}}), we obtain
\[\sum_{k=1}^{n-1}{N_{\classSn{ks},\carmult}} = q^{\frac{n+1}{2}-\frac{2m-m'}{2}} N_{1/\parameter^n}.\]
The left hand side is equal to $\multZnZinv{s} \sum_{\classSn{s'} \in \classSnZnZinv{s}}{N_{\classSn{s'},\carmult}}$ i.e. to $\frac{\multZnZinv{s}}{\multSn{s}} N_{\classSnZnZinv{s}}$. As $\classZnZ{s} \neq \classZnZ{0}$, \lmm{\ref{result:K.divides.gamma}} \pg{\pageref{result:K.divides.gamma}} shows that $\multSn{s}/\multZnZinv{s}$ is an integer. The result is hence proved.
\end{proof}

\begin{remark}
When $m' = m-2$, we have $v_{m-1}-w_{m-1} = w_m-v_m$ by \lmm{\ref{resultat:sum.v=sum.w}} \pg{\pageref{resultat:sum.v=sum.w}} and the equation of the variety simplifies greatly:
\[ H_{1/\parameter^n} \colon \begin{system}
y^n = x_1^{v_1} \dots x_{m}^{v_m} (1-x_1)^{w_1-v_1} \dots (1-x_m)^{w_m-v_m} \\
x_1 \dots x_m = \parameter^n 
\end{system}\]
\end{remark}

\subsection{Conclusion}\label{subsection:nb.points.link:conclusion}

We are now capable of showing \frml{\eqref{formula:AIM}} of the introduction. We begin by a result giving a lower bound on the number of pairings which will enable us to show that the dimension of the hypergeometric varieties is always $\leq n-4$.

\begin{proposition}\label{result:nb.pairings}
Let $\classSnZnZinv{s}$ be distinct from the class of $(0,1,\dots,n-1)$ and of $(0,\dots,0)$ and let $s$ be a representative of $\classSnZnZinv{s}$. We can choose two sequences $(v_1,\dots,v_m)$ and $(w_1,\dots,w_m)$ satisfying the assumptions of \rmrk{\ref{definition:v.w}} \pg{\pageref{definition:v.w}} such that we have the pairing
\[\forall i \in \integerinterval{1}{\tfrac{2m-n+1}{2}}, \quad w_{2i-1} - v_{2i-1} \equiv -(w_{2i} - v_{2i}) \mod n.\]
\end{proposition}

\begin{proof}
Let $(v_1,\dots,v_m)$ and $(w_1,\dots,w_m)$ be sequences as in \rmrk{\ref{definition:v.w}}. By \thrm{1.2} of \cite[\pg{126}]{Alon.additive.latin.transversals}, it is possible to permute $(w_1,\dots,w_m)$ so that the $v_i-w_i$ are pairwise distinct. Define $V$ as the subset $\set{v_i-w_i}$ of $\ZnZ$ (it has $m$ elements) and $\mu$ as the number of opposite pairs contained in $V$; we have
\[2\mu = \card{V \cap (-V)} = 2m - \card{V \cup (-V)} \geq 2m - (n-1).\]
As $2\mu$ is the maximal number of pairings, this ends the proof.
\end{proof}

\begin{theorem}\label{result:conclusion}
If $\parameter^n \neq 1$, we can write
\begin{align*}
\card{\dworkhypersurface(\Fq)} & = 1 + q + \dots + q^{n-2} + N_\textup{mirror} \\
& \quad + q^{\frac{n-3}{2}}N_1 + q^{\frac{n-5}{2}}N_3 + \dots + qN_{n-4},
\end{align*}
where each $N_d$ is a sum of terms of the form $\card{H_{\lambda}(\Fq)} - (q-1)^{l-1}q^{d+1-l}$ where $\alpha_i$ and $\beta_j$ are obtained from each $\classSnZnZinv{s}$ as described in \subsctn{\ref{subsection:nb.points.link:link}} and where each $H_{\lambda} \subset \varaff{d+2}$ is a variety of hypergeometric type of odd dimension equal to $d$ with $1 \leq d \leq n-4$ (here, $\lambda = 1/\parameter^n$) as considered in \subsctn{\ref{subsection:nb.points.hypergeom:formulas}}.
\end{theorem}

\begin{proof}
We saw in \thrm{\ref{result:nb.zeros:psi<>0.organise}} \pg{\pageref{result:nb.zeros:psi<>0.organise}} that, if $\parameter \neq 0$ and $q \equiv 1 \mod n$, we could write
\[\card{\dworkhypersurface(\Fq)} = 1 + q + \dots + q^{n-2} + N_{\classSnZnZinv{0}} + \sum_{\classSnZnZinv{s} \neq \classSnZnZinv{0}}{N_{\classSnZnZinv{s}}}.\]
In \thrm{\ref{result:nb.zeros.mirror.factor}}, we recalled Wan's result showing that $N_{\classSnZnZinv{0}} = N_\mathrm{mirror}$ and in \lmm{\ref{resultat:nb.zeros.terme.sing}}, we showed that the term corresponding to $(0,1,2,\dots,n-1)$ was zero when $\parameter^n \neq 1$.

Let us now consider $\classSnZnZinv{s}$ distinct from the class of $(0,\dots,0)$ and of $(0,1,2,\allowbreak\dots,\allowbreak n-1)$. Let $m'$ be the greatest even integer $\leq m-2$ such that there exists two sequences $(v_1,\dots,v_m)$ and $(w_1,\dots,w_m)$ as in \rmrk{\ref{definition:v.w}} verifying
\[\forall i \in \integerinterval{1}{\tfrac{m'}{2}}, \quad w_{2i-1} - v_{2i-1} \equiv -(w_{2i} - v_{2i}) \mod n.\]
By \prpstn{\ref{result:nb.pairings}}, we have $m' \geq 2m-n+1$ (note that, by \rmrk{\ref{remarque:propr.nprime:encadrement.m}} \pg{\pageref{remarque:propr.nprime:encadrement.m}}, $m+3 \leq  n$, hence $m-2 \geq 2m - n + 1$). The dimension $d = 2m-m'-3$ of the corresponding variety of hypergeometric type considered in \thrm{\ref{result:link.nb.points}} \pg{\pageref{result:link.nb.points}} thus satisfies $1 \leq d \leq n-4$.

Moreover, we have $q^{\frac{n+1}{2}-\frac{2m-m'}{2}} = q^{\frac{n-d-2}{2}}$, and so, as $d$ varies between $1$ and $n-4$, these powers of $q$ take the values $q^{\frac{n-3}{2}}$, \dots, $q$ respectively and all these values are obtained; indeed, if we consider an integer $m$ such that $2 \leq m = d+1 \leq n-3$ and define $s = (0,\dots,0,1,n-1,2,n-2,\dots,\frac{n-m-1}{2},\allowbreak n-\frac{n-m-1}{2})$, then $w = (0,\dots,0)$ and $v = (\frac{n-m+1}{2},n-\frac{n-m+1}{2},\dots,\frac{n-1}{2},\frac{n+1}{2})$ each consist of $m$ elements and we have $m' = m-2$ with the notations of \thrm{\ref{result:link.nb.points}}.
\end{proof}

\section{Examples}\label{section:nb.points.examples}

To illustrate the methods we have just presented in this paper, let's detail explicitly the cases $n=5$ and $n=7$; these examples are given in terms of the hypersurfaces of hypergeometric type from \subsctn{\ref{subsection:nb.points.hypergeom:hypersurface.hypergeom}}.

\begin{example}[$n=5$]\label{example:n=5} Let's recover the results announced by Candelas, de la Ossa and Rodriguez-Villegas in \cite{CdlORV.II} in the non-singular and non-diagonal case (see \cite{Goutet.quintique} for a complete treatment of the $n=5$ case). We are interested in the factorisation of the zeta function of the quintic $\quintichypersurface \colon x_1^5 + \dots + x_5^5 - 5 \parameter x_1 \dots x_5 = 0$ when $\parameter \neq 0$ and $\parameter^5 \neq 1$. We list all the classes $\classSnZnZinv{(s_1,\dots,s_5)} \neq \classSnZnZinv{(0,0,0,0,0)}$ and $\neq \classSnZnZinv{(0,1,2,3,4)}$ (following the notations from \subsctns{\ref{subsection:nb.points.dwork:prelim}, \ref{subsection:nb.points.link:K|gamma}, \ref{subsection:nb.points.link:prelim} and \ref{subsection:nb.points.link:link}}):
\begin{center}\begin{tabular}{cccccc}
$\classSnZnZinv{s}$ & $\multSn{s}$ & $\multZnZinv{s}$ & $m$ & $m'$ & $d$ \\
\noalign{\vspace{1pt}}
\hline
\noalign{\vspace{2pt}}
$(0,0,0,1,4)$ &   20 & 2 & 2 & 0 & 1 \\
$(0,0,1,1,3)$ &   30 & 2 & 2 & 0 & 1 \\
\end{tabular}\end{center}
Using the method described above, we obtain the following table (the hypergeometric hypersurfaces are all of the form $y^5 = x^{v_1}(1-x)^{v_2}(1-\frac{1}{\parameter^5}x)^{5-v_2}$).
\begin{center}\begin{tabular}{cccccl}
$\classSnZnZinv{s}$ & $v_1$ & $v_2$ & $w_1$ & $w_2$ & \textsc{equation}\\
\noalign{\vspace{1pt}}
\hline
\noalign{\vspace{2pt}}
$(0,0,0,1,4)$ & 2 & 3 & 0 & 0 & $y^5 = x^2(1-x)^3(1-\tfrac{1}{\parameter^5}x)^2$\\
$(0,0,1,1,3)$ & 2 & 4 & 0 & 1 & $y^5 = x^2(1-x)^4(1-\tfrac{1}{\parameter^5}x)$\\
\end{tabular}\end{center}
We find the same equations as those given in \cite[\sctn{11.1}]{CdlORV.II}:
\[\curveA \colon y^5 = x^2(1-x)^3(1-\tfrac{1}{\parameter^5}x)^2 \quad \text{and} \quad \curveB \colon y^5 = x^2(1-x)^4(1-\tfrac{1}{\parameter^5}x).\]
We set $N_{\curveA} = \card{\curveA(\Fq)}-q$ and $N_{\curveB} = \card{\curveB(\Fq)}-q$ (these number of points are affine). We have, when $\psi \neq 0$, $\psi^5 \neq 1$ and $q \equiv 1 \mod 5$:
\[\card{\quintichypersurface(\Fq)} = 1+q+q^2+q^3 + N_\textup{mirror} + 10qN_{\curveA} + 15qN_{\curveB}.\]
\end{example}

\begin{example}[$n=7$]\label{example:n=7}
We use the preceding results to find the factorisation of the zeta function of the septic $S_\parameter \colon x_1^7 + \dots + x_7^7 - 7\parameter x_1 \dots x_7 = 0$. We list the $\classSnZnZinv{(s_1,\dots,s_7)} \neq \classSnZnZinv{(0,\dots,0)}$ and $\neq \classSnZnZinv{(0,1,2,3,4,5,6)}$ (following the notations from \subsctns{\ref{subsection:nb.points.dwork:prelim}, \ref{subsection:nb.points.link:K|gamma}, \ref{subsection:nb.points.link:prelim} and \ref{subsection:nb.points.link:link}}):
\begin{center}\begin{tabular}{ccccccc}
$\classSnZnZinv{s}$ & $\multSn{s}$ & $\multZnZinv{s}$ & $m$ & $m'$ & $d$ \\
\noalign{\vspace{1pt}}
\hline
\noalign{\vspace{2pt}}
$(0,0,0,1,2,5,6)$ &  840 & 2 & 2 & 0 & 1 \\
$(0,0,1,1,3,4,5)$ & 1260 & 2 & 2 & 0 & 1 \\
$(0,0,1,1,2,4,6)$ & 1260 & 2 & 2 & 0 & 1 \\
\noalign{\vspace{2pt}}
\hline
\noalign{\vspace{2pt}}
$(0,0,0,0,1,2,4)$ &  210 & 3 & 3 & 0 & 3 \\
$(0,0,0,1,1,2,3)$ &  420 & 1 & 3 & 0 & 3 \\
$(0,0,1,1,3,3,6)$ &  630 & 3 & 3 & 0 & 3 \\
\noalign{\vspace{2pt}}
\hline
\noalign{\vspace{2pt}}
$(0,0,0,0,0,1,6)$ &   42 & 2 & 4 & 2 & 3 \\
$(0,0,0,0,1,1,5)$ &  105 & 1 & 4 & 2 & 3 \\
$(0,0,0,1,1,1,4)$ &  140 & 2 & 4 & 2 & 3 \\
$(0,0,0,1,1,6,6)$ &  210 & 2 & 4 & 2 & 3 \\
\end{tabular}\end{center}
The result is that, when $\parameter \neq 0$, $\parameter^7 \neq 1$ and $q \equiv 1 \mod 7$, the number of points takes the form
\[\card{S_\parameter(\Fq)} = 1+q+q^2+q^3+q^4+q^5+N_\textup{mirror} + q^2N_{1} + qN_3,\]
where the terms corresponding to curves of $\varaff{2}$ can be written as
\[N_{1} = 420N_{c_1} + 630N_{c_2} + 630N_{c_3},\]
and those corresponding to threefold hypersurfaces of $\varaff{4}$ can be written as
\[N_{3} = 70N_{t_1} + 420N_{t_2} + 210N_{t_3} + 21N_{t'_1} + 105N_{t'_2} + 70N_{t'_3} + 
105N_{t'_4},\]
where the various terms are defined in the following table (the corresponding number of points are in the affine space).

\begin{center}\begin{tabular}{lc}
\multicolumn{1}{l}{\textsc{equation of the hypersurface}} & \textsc{nb. of pts.} \\
\hline
\noalign{\vspace{2pt}}
$y^7 = x^3(1-x)^4 (1-\frac{1}{\parameter^7}x)^3$ & $q+N_{c_1}$\\
$y^7 = x^2 (1-x)^6 (1-\tfrac{1}{\psi^7} x)$ & $q+N_{c_2}$\\
$y^7 = x^3 (1-x)^5 (1-\tfrac{1}{\parameter^7} x)^2$ & $q+N_{c_3}$ \\
%
\noalign{\vspace{2pt}}
\hline
\noalign{\vspace{2pt}}
$y^7 = x_1^3 x_2^5 x_3^3 (1-x_1)^4 (1-x_2-x_3)^6 (1-\tfrac{1}{\parameter^7} x_1 x_2)$ & $q^3+N_{t_1}$\\
%
%
$y^7 = x_1^4 x_2^5 x_3^4 (1-x_1)^3 (1-x_2-x_3)^6 (1-\tfrac{1}{\parameter^7} x_1 x_2)$ & $q^3+N_{t_2}$\\
%
%
$y^7 = x_1^2 x_2^4 x_3^4 (1-x_1)^6 (1-x_2-x_3)^5 (1-\tfrac{1}{\parameter^7} x_1 x_2)^2$ & $q^3+N_{t_3}$\\
%
%
\noalign{\vspace{2pt}}
\hline
\noalign{\vspace{2pt}}
$y^7 = x_1^2 x_2^5 x_3^3 (1-x_1)^5 (1-x_2)^2(1-x_3)^4 (1-\tfrac{1}{\parameter^7}x_1x_2x_3)^3$ & $q^3+N_{t'_1}$\\
$y^7 = x_1^3 x_2^3 x_3^2 (1-x_1)^4 (1-x_2)^4 (1-x_3)^6 (1-\tfrac{1}{\parameter^7}x_1x_2x_3)$ & $q^3+N_{t'_2}$\\
$y^7 = x_1^3 x_2^5 x_3^2 (1-x_1)^4 (1-x_2)^3 (1-x_3)^6 (1-\tfrac{1}{\parameter^7}x_1x_2x_3)$ & $q^3+N_{t'_3}$\\
$y^7 = x_1^3 x_2^5 x_3^2 (1-x_1)^4 (1-x_2)^3 (1-x_3)^4 (1-\tfrac{1}{\parameter^7}x_1x_2x_3)^3$ & $q^3+N_{t'_4}$
\end{tabular}\end{center}

Let's justify for example the equation corresponding to $[0,0,0,0,0,1,6]$. We have $\{v_1,v_2,v_3,v_4\} = \{2,3,4,5\}$ and $w_1 = w_2 = w_3 = w_4 = 0$. Let's take, for example, $v_1 = 2$, $v_2 = 5$, $v_3 = 3$ and $v_4 = 4$ so that $w_1 - v_1 = -(w_2 - v_2)$ and $w_3 - v_3 = -(w_4 - v_4)$. For this choice, we have $m = 4$, $m' = m-2 = 2$ and the equation we obtain is
\[y^7 = x_1^2 x_2^5 x_3^3 (1-x_1)^5 (1-x_2)^2 (1-x_3)^4 (1-\tfrac{1}{\parameter^7}x_1x_2x_3)^3.\]
This is the equation corresponding to $N_{t'_1}$. The other equations follow in a similar way.
\end{example}

\begin{remark}
Using the same method, we could treat the cases $n = 11$, $n = 13$, etc. The only difficulty is practical, as the number of classes $\classSnZnZinv{(s_1,\dots,s_{n})}$ grows quickly with $n$.
\end{remark}

\section*{Acknowledgments}

I would like to thank my thesis advisor, J.~Oesterl\'e, for the numerous improvements he suggested concerning the text. I would also like to thank Surya Ramana for providing a reference concerning \prpstn{\ref{result:nb.pairings}}.


\end{document}